\documentclass[12pt]{amsart}
\usepackage{amsmath,amssymb}
\usepackage{color}
\theoremstyle{plain}
\newtheorem{thm}{Theorem}[section]

\newtheorem{cor}[thm]{Corollary}
\newtheorem{pro}[thm]{Proposition}
\newtheorem{ex}[thm]{Example}
\newtheorem{que}[thm]{Question}
\newtheorem{prob}[thm]{Problem}

\newtheorem{df}[thm]{Definition}
\newtheorem{rem}[thm]{Remark}

\newcommand{\ind}{\mathrm{ind}\,}

\begin{document}

\title{From homogeneity to discrete homogeneity}

\author{Vitalij A.~Chatyrko and Alexandre Karassev}

\begin{abstract} This is a survey of recent and classical results concerning various types of homogeneity, such as $n$-homogeneity, discrete homogeneity, and countable dense homogeneity. Some new results are also presented, and several problems are posed.  
\end{abstract}

\makeatletter
\@namedef{subjclassname@2020}{\textup{2020} Mathematics Subject Classification}
\makeatother



\maketitle

All spaces are supposed to be Tychonoff. For standard topological notions we refer to \cite{E}.
\section{Homogeneous spaces} 

\begin{df}\label{homogen} A topological space $X$ is called {\it homogeneous} if for any points $x, y$ of $X$ there exists a homeomorphism $f: X \to X$ such that $f(x) = y$.
\end{df}

 Note that any homogeneous space is {\it similar} around all its points from the topological point of view. 

\begin{ex}  The following spaces are not homogeneous:
\begin{itemize}
\item[(a)] Any space which has both isolated points and non-isolated points. 
\item[(b)] Any closed bounded interval of the real line $\mathbb R$.
\item[(c)] Any space $X$ which has points $p,q$ such that
$\ind_p X \ne \ind_q X$.
\end{itemize}
\end{ex}

\begin{ex} The following spaces are homogeneous.
\begin{itemize}
\item[(a)] Any discrete topological space.
\item[(b)] The real line $\mathbb R$ and its subspace $\mathbb Q$ of rational numbers.
\item[(c)] The n-dimensional sphere $S^n, n \geq 1$.
\item[(d)] The Sorgenfrey line $\mathbb S$ (the set of real numbers with the base of topology  provided by intervals $[a, b)$). 

\end{itemize}

\end{ex}
The following statement is obvious.

\begin{thm} If $\{X_\alpha, \alpha \in A\}$, is a family of homogeneous spaces then the topological product
$\prod_{\alpha \in A} X_\alpha$ is also a homogeneous space.
\end{thm}

\begin{ex} The following are examples of homogeneous spaces which are homeomorphic to products of homogeneous spaces.
\begin{itemize}
\item[(a)] The space $\mathbb P$ of irrational numbers which is homeomorphic to $Z^\omega$, where $Z$ is the space of integer numbers.
\item[(b)] The Cantor set $\mathbb C$ (since $\mathbb C$ is homeomorphic to $D^\omega$, where $D$ is the two-points discrete space).
\item[(c)] The topological products $\mathbb Q \times \mathbb C$, $\mathbb Q \times \mathbb P$.

\end{itemize}
\end{ex}

\begin{ex} The following spaces are also homogeneous:
\begin{itemize}

\item[(a)] The  Hilbert cube $[0,1]^\omega$.
\item[(b)] Erd\"os spaces $E(\mathbb Q)$ and $E(\mathbb P)$ 
(where $E(\mathbb X) = \{\overline{x} \in l_2: x_i \in \mathbb X \subset \mathbb R\}$ and $l_2$ is the real separable Hilbert space).
\item[(c)] Any topological group $G$. 

\end{itemize}
\end{ex}

Note that the spaces $\mathbb S$, $[0,1]^\omega$, and $S^n, n \ne 1,3$, are not topological groups unlike to other examples mentioned above.

\section{Generalizations of homogeneity} 

\begin{rem}\label{remark} One can substitute in Definition \ref{homogen} the points $x, y$ with the sets $A, B$ from a given non-empty class $\mathcal C$ of subsets of $X$, and require the homeomorphism $f\colon X\to X$ to be such that $f(A) = B$.

Note the following natural restrictions on the sets $A$ and $B$ for such a homeomorphism to exist: $A \simeq B$ and $X \setminus A \simeq X \setminus B$, where $\simeq$ means homeomorphic. In particular, 
$|A| = |B|$ and $|X \setminus A| = |X \setminus B|.$
\end{rem}

Further, we will discuss stronger forms of homogeneity based on Remark \ref{remark}.

Let $n$ be a positive integer. In 1954 Burgess \cite{Bu} called a topological space $X$ {\it $n$-homogeneous}  if
for any two  subsets $A$ and $B$ of $X$ of cardinality $n$ there exists a homeomorphism $f\colon X \to X$, such that $f(A)=B$.  Note that the $1$-homogeneous spaces are  homogeneous spaces in the sense of Definition~\ref{homogen}.  Further,  he called $X$ {\it strongly $n$-homogeneous} if
any bijection  $f\colon A \to B$ between  subsets $A$ and $B$ of $X$ of cardinality $n$ extends to a homeomorphism  $f\colon X \to X$. 

\begin{pro} Let $n$ be an integer $> 1$ and $X$ a topological space. Then the following is valid.
\begin{itemize}
\item[(a)] If $X$ is strongly $n$-homogeneous then $X$ is strongly $(n-1)$-homogeneous.
\item[(b)] If $X$ is strongly  $n$-homogeneous then $X$ is  $n$-homogeneous.
\item[(c)] If $X$ is $n$-homogeneous then $X$ is  $1$-homogeneous, i.e. homogeneous (see \cite[Theorem 1]{Bu}).

\end{itemize}

\end{pro}

\begin{que}  Are there an integer $n \geq 3$ and a $n$-homogeneous space which is not $2$-homogeneous?
\end{que}

\begin{ex}\label{n-homogeneous}
 \begin{itemize}
\item [(a)] The real line $\mathbb R$ is $n$-homogeneous for every integer $n \geq 1$, strongly 
$2$-homogeneous, and not strongly $3$-homogeneous.

\item [(b)] The circle $S^1$ is $n$-homogeneous for every integer $n \geq 1$, strongly 
$3$-homogeneous, and not strongly $4$-homogeneous.

\item [(c)] Any connected $k$-dimensional manifold (in particular,  the Euclidean space  $\mathbb R^k$ and the sphere $S^k$), $k \geq 2,$ is strongly $n$-homogeneous for every integer $n \geq 1$
(see \cite[Theorem 14]{Ba}).
\end{itemize}
\end{ex}

Let us add that 
if $X$ is an $n$-homogeneous compact metric connected space then $X$ is strongly 
$n$-homogeneous or it is the circle $S^1$ \cite[Theorem 3.11]{U}, and
there is a connected Lindel\"of space that is $n$-homogeneous for every positive integer $n$ and not strongly $2$-homogeneous \cite[Example]{vM2}.

\begin{que} Is there strongly $4$-homogeneous space which is not strongly $5$-homogeneous?
\end{que}

\begin{ex} The topological union $\mathbb R \oplus \mathbb R$ is $1$-homogeneous
but  not $2$-homogeneous.
\end{ex}

\begin{que} Is there $2$-homogeneous space which is not $3$-homogeneous?
\end{que}

It is well-known that any finite dimensional metrizable locally connected locally compact group is a Lie group \cite{MZ} and hence if it is connected and has dimension $\geq 2$ then it is strongly $n$-homogeneous for any $n $. However, the answer to the following question is unknown to the authors.

\begin{que} Are there connected and locally connected topological groups which are not $n$-homogeneous for some $n>1$? 
\end{que}

\section{A generalization of the (strong) $n$-homogeneity} 
A family of subsets $\mathcal D$ of a topological space $X$ is 
{\it discrete} if each point $x \in X$ has a neighborhood  $Ox$ which  intersects at most one  set from $\mathcal D$. If each element of $\mathcal D$ is a singleton, we obtain the definition of {\it a discrete subset} of $X$
as the union of these singletons. 

Note that each discrete subset of $X$ is closed, and any finite subset of $X$ is discrete.

In 2022 the authors \cite{ChK1} suggested the following notions.

A Hausdorff topological space $X$ is called 
\begin{itemize}
\item {\it discrete homogeneous} ($DH$), if for any two discrete subsets $A$ and $B$ of $X$ with $|A|=|B|$ there exists a homeomorphism 
$f$ of $X$ onto itself such that $f(A)= B$;

\item {\it strongly discrete homogeneous} ($sDH$), if for any two discrete subsets $A$ and $B$ of $X$ and any bijection $f\colon A\to B$ this bijection extends to a homeomorphism of the whole $X$ onto itself.

\end{itemize}

Some simple facts from \cite{ChK1} are mentioned next.

\begin{pro} Let $X$ be a topological space. Then the following is valid.
\begin{itemize}
\item[(a)] If $X$ is a $sDH$ space then $X$ is a $DH$ space.
\item[(b)] If $X$ is a $sDH$ space then $X$ is  strongly $n$-homogeneous for every $n \geq 1$.
\item[(c)] If $X$ is a $DH$ space then $X$ is  $n$-homogeneous for every $n \geq 1$.

\end{itemize}

\end{pro}

\begin{ex} 
\begin{itemize}
\item[(a)] Any infinite discrete space  is strongly $n$-homogeneous for every $n \geq 1$ but  not a $DH$-space, and hence  not an $sDH$ space.
\item[(b)] The real line $\mathbb R$  is $n$-homogeneous for every $n \geq 1$ but  not a $DH$-space. 

\end{itemize}

\end{ex}

\begin{pro}\label{compact_sDH} Let $X$ be a compact Hausdorff space. Then:

\begin{itemize}
\item[(a)] $X$ is  a $sDH$-space iff it is strongly $n$-homogeneous for every $n \geq 1$;
\item[(b)] $X$ is  a $DH$-space iff it is $n$-homogeneous for every $n \geq 1$.
\end{itemize}

\end{pro}

\begin{ex} The circle $S^1$  is  $DH$  but  not  $sDH$.
\end{ex}

\begin{que} Does  there exist a noncompact metric space which is  $DH$  but not $sDH$?
\end{que}

\section{Zero-dimensional $sDH$ spaces}
In 1981 van Mill introduced {\it the strongly homogeneous spaces} as those all nonempty clopen subspaces of which are homeomorphic. Any connected space  is evidently strongly homogeneous. Therefore this concept is of interest only in the case of disconnected spaces.

The following  result was obtained by van Mill.

\begin{thm}(\cite[Theorem 3.1]{vM1}) Every homeomorphism between two closed nowhere dense subsets of a strongly homogeneous separable metrizable zero-dimensional space $X$ can be extended to an automorphism  of $X$. In particular, the space $X$  is homogeneous.
\end{thm}

Note that the spaces $\mathbb C$, $\mathbb Q$,  $\mathbb P$, ${\mathbb C}\times\mathbb Q$, ${\mathbb P}\times\mathbb Q$ are strongly homogeneous \cite{vM1}, separable metrizable and zero-dimensional. Since each discrete subset of a space without isolated points is closed and nowhere dense we get that the spaces $\mathbb C$, $\mathbb Q$,  $\mathbb P$, ${\mathbb C}\times\mathbb Q$, ${\mathbb P}\times\mathbb Q$ are sHD. 

In 2022 the authors \cite{ChK1} defined {\it a property $\mathcal D$} as follows. 
A Hausdorff space $X$ possesses the property $\mathcal D$ if for any proper clopen homeomorphic subspaces $F$ and $G$ of $X$ the subspaces $X \setminus F$ and $X \setminus G$ are also homeomorphic.

Note that the spaces $\mathbb C$, $\mathbb Q$,  $\mathbb P$, ${\mathbb C}\times\mathbb Q$, ${\mathbb P}\times\mathbb Q$ and $\mathbb S$ have the property $\mathcal D$. 
Observe also that all mentioned above spaces are paracompact, zero-dimensional in the sense of dimension $\ind$ (i.e. they have topology bases consisting of clopen sets) and homogeneous.
Furthermore, there are other zero-dimensional in the sense of dimension $\ind$, paracompact,  homogeneous spaces possessing  property $\mathcal D$ (one of them is a van Douwen's subgroup $H$ of the circle group from \cite{vD} which is not strongly homogeneous).

Recall  that a topological space $X$ is called {\it collectionwise normal} if $X$ is $T_1$ and for every discrete family $\{F_s\}_{s \in S}$ of closed subsets of $X$ there exists a discrete family $\{V_s\}_{s \in S}$  of open subsets of $X$ such that $F_s \subset V_s$ for every $s \in S$. 

It is a well-known fact that every paracompact space is collectionwise normal.

\begin{thm}(\cite[Theorem 4.2]{ChK1}) Let $X$ be a collectionwise normal zero-dimensional in the sense of $\,\ind$ homogeneous topological space with  the property $\mathcal D$. Then $X$ is $sDH$.

In particular, the spaces  $\mathbb S$ and $H$ are sDH. 
\end{thm}

Since  each zero-dimensional in the sense of dimension $\ind$ first countable strongly homogeneous space is homogeneous, we have the following.

 \begin{cor} (\cite[Corollary 2.7]{ChK1}) Let $X$ be a collectionwise normal first countable strongly homogeneous topological space with $\ind X =0.$ Then $X$ is sDH.
\end{cor}

Note that $\mathbb X^n \simeq \mathbb X$, whenever  $X$ is $\mathbb C$, $\mathbb Q$,  $\mathbb P$, ${\mathbb C}\times\mathbb Q$, ${\mathbb P}\times\mathbb Q$ and $n$ is any poitive integer. 
Since $\mathbb S^n$ is not normal for $n > 1$, $\mathbb S^n$ is not homeomorphic to $\mathbb S$ for $n > 1$.

\begin{pro}(\cite[Proposition 4.4]{ChK1}) For each $n > 1$ the space $\mathbb S^n$ is not  $DH$. 
\end{pro}

We do not know if $H \simeq H^2$.
\begin{que} 
Are  $H^n$ $(s)DH$ for each positive integer $n$?
\end{que}

Note that $\mathbb C \simeq \mathbb C^\omega$ and $\mathbb P \simeq \mathbb P^\omega$ but $\mathbb Q^\omega$
is not homeomorphic to $\mathbb Q$.

\begin{que} 
Are  $\mathbb Q^\omega$, $H^\omega$  $(s)DH$?
\end{que}

Since $\mathbb Q^\omega$, $H^\omega$ and $H^n, n \geq 2$, are collectionwise normal, zero-dimensional in the sense of $\ind$ and homogeneous, the following question naturally arises.

\begin{que} 
Do $\mathbb Q^\omega$, $H^\omega$, or $H^n, n \geq 2$, possess the property $\mathcal D$?
\end{que}

The following notion from \cite{ChK1} was motivated by Theorem 4.1. 

Let $X$ be a topological space and the class of all closed nowhere dense subsets of $X$ is not empty. 
The space $X$ is called
 {\it nowhere dense homogeneous} (NDH), if for any two closed nowhere dense subsets $F$ and $G$ of $X$ and any homeomorphism  $f\colon F \to G$ there exists an automorphism $\bar{f}$ of $X$ which extends $f$.

Theorem 4.1 implies the following.
\begin{pro} Every strongly homogeneous separable metrizable zero-dimensional space is NDH.
\end{pro}

\begin{cor}\label{cor_NDH} The spaces $\mathbb C$, $\mathbb Q$,  $\mathbb P$, ${\mathbb C}\times\mathbb Q$, ${\mathbb P}\times\mathbb Q$ are NDH. 	
\end{cor}

The following proposition together with Corollary \ref{cor_NDH} gives an answer to \cite[Question 2.13 (a)]{ChK1}.

\begin{pro} Let $X$ be an NDH topological space and $D$ be a non-empty discrete topological space. Then the topological union 
$Y = X \oplus D$ of $X$ and $D$ is NDH but it is neither homogeneous nor strongly homogeneous. In particular, $Y$ is not (s)DH.  
\end{pro}
\begin{proof} Note that the class of all closed nowhere dense subsets of $Y$ coincides with the class of all closed nowhere dense subsets of $X$.
\end{proof}

\begin{que} Are the Sorgenfrey line $\mathbb S$ or the van Douwen's group  $H$ NDH? 
\end{que}

It is easy to see that each NDH space without isolated points is sDH. So the following fact is valid.

\begin{pro} Let $X$ be a space without isolated points. If $X$ is not sDH then $X$ is not NDH.
In particular, the Sorgenfrey powers $\mathbb S^n, n \geq 2, $ the real line $\mathbb R$, the circle $S^1$ are not NDH.
\end{pro}

\section{Non-zero-dimensional $sDH$ spaces}

 \begin{pro} (\cite[Propositions 2.14 and 2.15]{ChK1}) Every finite-dimensional Euclidean space $\mathbb R^n, n >1,$ is 
 an  sDH space.
\end{pro}

\begin{pro}\label{nonNDH} Every finite-dimensional Euclidean space $\mathbb R^n, n \geq 2,$  is not NDH.
\end{pro}
\begin{proof}
To see that consider the closed nowhere dense subset $X = (\{0\} \cup \{1\} \cup \{2\})  \times \mathbb R^{n-1}$ of $\mathbb R^n$
and a homeomorphism $f : X \to X$ such that $f(0, \overline{x}) = (0, \overline{x})$,
$f(1, \overline{x}) = (2, \overline{x})$ and $f(2, \overline{x}) = (1, \overline{x})$ for any  $\overline{x} \in \mathbb R^{n-1}$. It is easy to  see that $f$ cannot be extended to an automorphism of $\mathbb R^n$.
In fact, if such automorphism $\bar{f} : \mathbb R^n \to \mathbb R^n$ exists and $I$ is a bounded closed interval connecting the sets $\{0\}\times R^{n-1}$ and $\{1\}\times R^{n-1}$
then $\bar{f}(I) \cap \bar{f}(\{2\}\times R^{n-1}) \ne \emptyset$. We have a contradiction.
\end{proof}

\begin{pro} (\cite[Proposition 2.16]{ChK1}) The separable Hilbert space $l_2$ is sDH.
\end{pro}

\begin{pro}  The separable Hilbert space $l_2$ is not NDH.
\end{pro}
\begin{proof} Represent $l_2$ as $\mathbb R\times l_2$ and proceed as in the proof of Proposition~\ref{nonNDH}.
\end{proof}

Recall that a space $X$ is called {\it strongly locally homogeneous (SLH)} if it has a basis of open sets $\mathcal B$ such that for every $U\in\mathcal B$ and any two points $x$ and $y$ in $U$ there exists a homeomorphism $f\colon X\to X$ such that $f(x) =y$ and $f$ is identity on $X\setminus U$.  Note that any $n$-manifold for $n>1$, any Hilbert space manifold, and any Hilbert cube manifold  are examples of SLH spaces. 

In 1976 Bales \cite{Ba} showed that  a connected SLH space, no two-point subset of which disconnects it, is strongly $n$-homogeneous for any $n$.

\begin{ex}\label{nondh}(\cite[Proposition 2.17]{ChK1}) Let $M = \mathbb R\times S^1.$ Then $M$ is a $2$-manifold which is strongly $n$-homogeneous for any $n$ and SLH but not DH.
\end{ex}

The following question arises.

 \begin{que}\label{manifold_question} Which metrizable manifolds are DH or sDH? 
 \end{que}
 
 We will answer this question in the next section.

  \begin{prob} Find conditions for SLH space to be (s)DH. 
 \end{prob}

Recall that a closed subset $A$ of a metric space $X$ is called a $Z$-set if for any $\epsilon >0$ there exists a map $f\colon X \to X$ such that $f$ is $\epsilon$-close to the identity map and $f(X)\cap A = \varnothing$. 

Let $\mu^k$ denote the Menger cube for each $k=1,2,\dots$, and $Q=[0,1]^\omega$ the Hilbert cube.  It is well-known that any finite subset of $\mu^k$ or $Q$ is a $Z$-set in the respective space, and any homeomorphism between two $Z$-sets of $\mu^k$ or $Q$ can be extended to a homeomorphism of the whole space (see, e.g. \cite{Ch} for the case of the Hilbert cube, and \cite{A, Be} for the case of Menger compacta). Consequently, $Q$ and $\mu^k$ are examples of sDH compacta.

Moreover, it is well-known  that $\mu^1$ and $S^1$ are the only homogeneous $1$-dimensional Peano continua (by a continuum we mean metrizable compact connected space). By the result of Ungar \cite{U}, any $2$-homogeneous continuum is locally connected, i.e. Peano. Therefore we have the following observation.

\begin{pro}(\cite[Proposition 2.22]{ChK1})
$\mu^1$ and $S^1$ are the only DH $1$-dimensional continua.
\end{pro}

Since as was observed earlier $S^1$ is not sDH but $\mu^1$  is sDH, we get immediately the following corollary.

\begin{cor}(\cite[Corollary 2.23]{ChK1}) $\mu^1$ is the only sDH $1$-dimensional continuum
\end{cor}

Similarly to Proposition 5.2 we can prove the following.

\begin{pro} $\mu^1$ is not NDH.
\end{pro}

\begin{cor} There is no $1$-dimensional continuum which is NDH.
\end{cor}

\begin{que} Is there a separable metrizable non-zero dimensional NDH space?
\end{que}

Recalll that Erd\"os spaces $E(\mathbb Q)$ and $E(\mathbb P)$ are one-dimensional spaces that are homeomorphic to their own squares.

\begin{que} 
Are  Erd\"os spaces $E(\mathbb Q)$ and $E(\mathbb P)$  (strongly) $n$-homogeneous for every integer $n \geq 2$?
\end{que}

\begin{que} 
Are  Erd\"os spaces $E(\mathbb Q)$ and $E(\mathbb P)$  $(s)DH$ or NDH?
\end{que}

\section{Charakterization of $sDH$ (DH) connected metrizable manifolds of dimension $\geq 2$.}

Note that every connected compact metrizable $k$-manifold, $k \geq 2$, is $sDH$ (see Example~\ref{n-homogeneous} and Proposition~\ref{compact_sDH}). 

In 2023 the authors \cite{ChK2} introduced the following notions.

A space is called {\it generalized continuum} if it is locally compact, $\sigma$-compact, non-compact, connected, and locally connected Hausdorff space.

Let $X$ be a generalized continuum. 
There exists a sequence (generally speaking, not unique) of compact subsets $\{ X_n\}_{n=1}^\infty$ of $X$ such that $X_n\subset {\rm int} X_{n+1}$ for all $n=1,2, \dots$  and $\cup_{n=1}^\infty X_n = X$.

Further,
by {\it an end} of $X$ with respect to $\{ X_n\}_{n=1}^\infty$ we mean a sequence $\{ Y_n\}_{n=1}^\infty$ such that for all $n$  the set $Y_n$ is a non-empty connected component of $X\setminus X_n$ and $Y_{n+1}\subset Y_n$.

Note that the Euclidean plane $\mathbb R^2$ as well as the product $S^1 \times \mathbb R$ are generalized 
continua. Moreover,  the plane  $\mathbb R^2$ has one end and the space  $S^1 \times \mathbb R$ has two ends.

This distinction can be extended to the following fact.

\begin{thm} (\cite[Theorem 4.3]{ChK2}) A connected non-compact metrizable manifold of dimension $n \geq 2$ is $sDH$ ($DH$) iff
it has one end.
\end{thm}

\begin{prob} Let $X$ be a metrizable strongly locally homogeneous generalized continuum with one end. Is $X$  (s)DH?
	\end{prob}
\section{Other generalization of homogeneity}

A separable space $X$ is called CDH if for any two countable dense subsets $A, B$ of $X$ there exists a homeomorphism $f: X \to X$ such that $f(A) = B$. 

The CDH-spaces do not need to be homogeneous (for example, the topological sum $\mathbb R \oplus \mathbb P $) but if a CDH space is connected then it is homogeneous. In 2012 van Mill \cite{vM2} proved that connected CDH spaces are even n-homogeneous for each positive integer $n$, and there exists a connected Lindel\"of CDH space $M$ which is not strongly $2$-homogeneous.

\begin{que} 
Is the space $M$ $DH$?
\end{que}

In 1975 Ungar \cite{U} proved for a locally compact separable metrizable space $X$ such that no finite set separates it the equivalence of the following conditions:
\begin{itemize}
\item[(a)] $X$ is CDH,
\item[(b)] $X$ is $n$-homogeneous for every positive integer $n$,
\item[(c)] $X$ is strongly $n$-homogeneous for every positive integer $n$.
\end{itemize}

\vskip 0.3 cm

\noindent(V.A. Chatyrko)\\
Department of Mathematics, Linkoping University, 581 83 Linkoping, Sweden.\\
vitalij.tjatyrko@liu.se

\vskip 0.3 cm
\noindent(A. Karassev) \\
Department of Mathematics,
Nipissing University, North Bay, Canada\\
alexandk@nipissingu.ca

\end{document}